\documentclass[11pt,letterpaper]{amsart}

\usepackage{amsmath}
\usepackage{amssymb}
\usepackage{amscd}
\usepackage{latexsym}
\usepackage[utf8]{inputenc}

\usepackage{graphicx}
\usepackage{changebar,color}
\usepackage{mathrsfs}
\usepackage[colorlinks=true, linkcolor=blue, citecolor=blue]{hyperref}

\usepackage{tikz-cd}
\usetikzlibrary{calc}

\tikzset{
  curve/.style={
    settings={#1},
    to path={
      (\tikztostart)
      .. controls
      ($(\tikztostart)!\pv{pos}!(\tikztotarget)!\pv{height}!270:(\tikztotarget)$)
      and
      ($(\tikztostart)!1-\pv{pos}!(\tikztotarget)!\pv{height}!90:(\tikztostart)$)
      .. (\tikztotarget)\tikztonodes
    },
  },
  settings/.code={
    \tikzset{quiver/.cd,#1}
    \def\pv##1{\pgfkeysvalueof{/tikz/quiver/##1}}
  },
  quiver/.cd,
  pos/.initial=0.35,
  height/.initial=0
}
\usepackage{array}

\newtheorem{theorem}{Theorem}[section]
\newtheorem{corollary}[theorem]{Corollary}
\newtheorem{lemma}[theorem]{Lemma}

\newtheorem{proposition}[theorem]{Proposition}

\newtheorem{definition}[theorem]{Definition}

\newtheorem*{manualtheorem}{Theorem}

\numberwithin{equation}{section}

\def\Z{\mathbb{Z}}

\def\Q{\mathbb{Q}}
\def\R{\mathbb{R}}

\def\F{\mathbb{F}}

\date{\today}

\title[A surface with infinite étale fundamental group]{A proper arithmetic surface with an infinite étale fundamental group}
\author{Pablo Chenal}
\address{Pablo Chenal, Universit\'e Paris-Saclay, Laboratoire de Math\'ematiques d'Orsay, 91405 Orsay Cedex, France, and  \'Ecole Normale Sup\'erieure, DMA, 75230 Paris Cedex, France.}
\email{pablo.chenal@universite-paris-saclay.fr}

\begin{document}

\begin{abstract}
We give an example of a proper, regular, integral arithmetic surface over $\mathrm{Spec}\ \Z$ endowed with a section, and whose étale fundamental group admits an infinite pro-$2$ quotient. The construction is a geometric analog of the Golod--Shafarevich towers for rings of integers of number fields.
\end{abstract}

%
%
\maketitle

\section{Introduction}

\subsection{Main theorem}
We prove the following theorem.
\begin{theorem}\label{maintheorem}
    There exists a proper, regular, integral arithmetic surface $f:X\rightarrow \mathrm{Spec}\ \Z$ admitting a section and such that $\pi^{\text{ét}}_1(X,*)$ has an infinite pro-$2$ quotient.
\end{theorem}

By an arithmetic surface $f:X\rightarrow \mathrm{Spec}\ \Z$, we mean a flat, separated scheme of finite type over $\mathrm{Spec}\ \Z$ of relative dimension $1$. $*$ denotes an arbitrary geometric point of $X$.
\subsection{Some context}
In \cite{BostCharles2026}, Bost and Charles obtained finiteness results for the étale fundamental group of arithmetic surfaces that satisfy some positivity conditions, the so-called \emph{pseudo-concave} arithmetic surfaces. These results apply, in particular, to models of elliptic curves and models of modular curves.

They raised the question of the existence of a regular, integral arithmetic surface over $\mathrm{Spec}\ \Z$ admitting a section and having an infinite étale fundamental group; see the discussion following Corollary 9.3.8 in \cite[Section 9.3.4]{BostCharles2026}. Theorem \ref{maintheorem} answers this question and shows that the surface can moreover be chosen proper.

In dimension $1$, Golod and Shafarevich constructed such examples using an \emph{arithmetic} tower 
$$\mathcal{O}_{K_1}\subset \mathcal{O}_{K_2}\subset ...\subset \mathcal{O}_{K_n}\subset ...$$
of rings of integers in a tower of number fields whose successive extensions are $p$-extensions unramified at every finite place. The criterion is that the number of relations $r$ of the Galois group $G_p$ of the maximal unramified $p$-extension of the number field $K_1$ is strictly smaller than $d^{2}/4$, where $d$ is the number of generators of $G_p$.

One might try to construct such an arithmetic surface using a fibration in curves over these \emph{arithmetic} towers. Our construction instead uses a \emph{geometric} tower. We use the Golod--Shafarevich criterion on the étale fundamental group $$\pi^{\text{ét}}_{1}(C,*)$$ of a curve $C$ over $\Q$, which gives us an infinite pro-$2$ quotient coming from the étale fundamental group of the reduction at the prime $2$ of this curve. This implies that we have a \emph{geometric} tower 
$$...\rightarrow C_n\rightarrow ...\rightarrow C_2\rightarrow C_1=C$$ of finite étale covers of $2$-power degree over $\Q$, which we extend to a tower of finite étale covers of $2$-power degree of arithmetic surfaces 
$$...\rightarrow \mathcal{C}_n\rightarrow ... \rightarrow \mathcal{C}_2\rightarrow \mathcal{C}$$
where $\mathcal{C}$ is a regular, integral, flat model of $C$ over $\mathrm{Spec}\ \Z$ endowed with a section.

The finiteness results of Bost--Charles \cite{BostCharles2026}
are obtained by Arakelov-theoretic methods. Our proof uses
specialization at the prime $2$ and control of the Galois action
on the geometric fundamental group. In this respect, it is closer
in spirit to the specialization and Galois coinvariant arguments
of Katz--Lang \cite{KatzLang81} in geometric class field theory.
Here, Golod--Shafarevich theory provides the necessary control
of non-abelian quotients.

The curve $C$ from which our construction starts has large genus
$g=2^{256}-1$. The results of Bost--Charles apply to models of
elliptic curves and to suitable integral models of modular curves,
whose genera are unbounded. Thus, a large genus does not by itself
force the étale fundamental group of an arithmetic surface
to be infinite.

These questions are also related to an observation of T.~Saito,
recorded in the appendix to Ihara's article
\cite[Appendix]{Ihara1994HorizontalDivisors}: if $D$ is a reduced union of sections
of $\mathbb{P}^{1}_{\Z}$ and is a normal crossings divisor, then
\[
\pi_1^{\text{ét}}(\mathbb{P}^{1}_{\Z}\setminus D)=1.
\]
In \cite{ShimizuYamaguchi2026}, Shimizu--Yamaguchi investigate
the converse and the role of the prime $2$. They prove that
this group admits a quotient isomorphic to $\Z/2\Z$ whenever
$D$ is not normal crossings at $2$, and ask whether it is
then infinite \cite[Remark~3.8]{ShimizuYamaguchi2026}.
These examples are non-proper arithmetic surfaces
whose generic fibers have genus zero.
\subsection{Related results}
In the preprint \cite{KellerPagano} shared with the author by Timo Keller and Carlo Pagano, there is a similar but more general construction of arithmetic surfaces with an infinite étale fundamental group. The same base curve $C$ is used, but instead of obtaining a proper arithmetic surface by taking a finite étale cover of $C$ (which is not explicit), they produce an explicit finite étale cover $C_0$ of $C$: this implies that they have an explicit infinite tower of degree $2$ étale covers of $C_0$.

Moreover, in their paper, they control the bad reduction fibers of the arithmetic surface above all bad reduction primes, while we only control the fiber above $2$ in our example. With this control, they obtain a comparison between the conductor and the genus of the family of curves appearing in the infinite tower of degree $2$ étale covers of $C_0$ that answers a question of Venkatesh.

Timo Keller, Carlo Pagano, and the author are working towards a unification of the two approaches,
and plan to merge the two papers into a single joint work.

\subsection{Organization of the article}
Section $2$ recalls the group-theoretic form of the Golod--Shafarevich theorem and proves two auxiliary lemmas used in Sections $3$ and $6$. In Section $3$, we prove a criterion for constructing \emph{geometric} towers: this is the main idea of this paper. In Section $4$, we define an explicit curve $C$ of large genus $g$ satisfying this criterion. In Section $5$, we extend the covers of the generic fiber of the proper, regular, minimal model of $C$ to an open subscheme of this model, flat over $\mathrm{Spec}\ \Z$. In Section 6, we adapt the construction to provide an example that is proper over $\mathrm{Spec}\ \Z$.
\subsection{Acknowledgments}
I am really grateful to my advisors, Jean-Benoît Bost and François Charles, for suggesting this problem to me and for sharing their ideas. I thank David Harari for a useful discussion about étale fundamental groups and, indirectly, Tamás Szamuely.

I thank Timo Keller and Carlo Pagano for sharing a draft version of their work, which allowed a simultaneous release.

\subsection{AI disclosure}
This example of arithmetic surface was produced through a discussion with ChatGPT6 Astra. The paper was written by the author, who is responsible for its correctness. 

\section{A brief review of the Golod--Shafarevich theorem and pro-2 groups}
\subsection{The Golod--Shafarevich theorem}
In \cite{GolodShafarevich1964Eng}, Golod and Shafarevich provided a criterion on the minimal number of generators and relations for a pro-$p$ group to be infinite, where $p$ is a prime. A pro-$p$ group $G$ is an inverse limit of $p$-groups, hence a profinite group. Denote by $r(G)$ the minimal number of relations and by $d(G)$ the minimal number of generators of $G$.

\begin{manualtheorem}[Golod--Shafarevich, 1964]
 Let $G$ be a non-trivial finite pro-$p$ group. Then we have the following inequality on the number of relations and generators $$r(G)>\left(\frac{d(G)-1}{2}\right)^{2}.$$

In particular, if $G$ is a non-trivial pro-$p$ group with $r(G)\leq \left(\frac{d(G)-1}{2}\right)^{2}$, then it is infinite. 
\end{manualtheorem}

The following refined version was proven by Vinberg \cite{Vinberg1965}, and we will use it in Section $6$ to prove Theorem \ref{maintheorem} in the proper case.

\subsubsection{Golod--Shafarevich polynomial}

Let $F$ be a free pro-$p$ group of rank $d$, where $p$ is a prime. Let $I$ be the ideal of $\mathbb{F}_p[[F]]$ defined as $$I:=\mathrm{Ker}(\mathbb{F}_p[[F]]\rightarrow \mathbb{F}_p).$$
Define the \emph{depth} of $x\in F\setminus \{e\}$ as $\omega(x):=\text{max}\{n,\ x-1\in I^{n}\}$. By convention, $\omega(e):=+\infty$. For $u\in F$ and $n$ a positive integer, we have the inequality
\begin{equation}\label{stabilitédegré}
\omega(u^{p^{n}})\geq p^{n}\omega(u).\end{equation}

Let $G$ be a finitely presented pro-$p$ group, and let
$$1\rightarrow R\rightarrow F\rightarrow G\rightarrow 1$$ be a minimal presentation of $G$ where $F$ is a free pro-$p$ group of rank $d$ and $R$ is a normal subgroup of $F$ topologically normally generated by $(R_i)_{1\leq i\leq r}$. For $k\geq 1$, define 
$$r_k:=\text{Card}\{R_i\in R,\ \omega(R_i)=k\} .$$

\begin{definition}
    The Golod--Shafarevich polynomial $P_G$ associated to this presentation of $G$ is 
    $$P_G(X):=1-dX+\overset{+\infty}{\underset{k=1}{\sum}} r_kX^{k}\in \R[X].$$
\end{definition}

\subsubsection{Vinberg's refined criterion}

We have the following refined version of the Golod--Shafarevich criterion, proven by Vinberg \cite{Vinberg1965}.
\begin{manualtheorem}[Vinberg, 1965]\label{vinberg}
    Let $G$ be a finite pro-$p$ group. Then, we have the inequality
    $$\forall \ 0<t<1,\ P_G(t)>0.$$
\end{manualtheorem}

We will use the following lemma in Section $6$, to deduce a proper arithmetic surface with infinite étale fundamental group from the non proper example of Section $5$.

\begin{lemma}\label{powersofinertia}
Let $G$ be a finitely presented pro-$2$ group with a presentation
whose Golod--Shafarevich polynomial satisfies $P_G(t)<0$ at some
$t$ with $0<t<1$. Let $\tau_1,\ldots,\tau_b\in G$. For $n$ large enough,
the quotient
\[
G_M:=G/\overline{\langle\!\langle
\tau_1^M,\ldots,\tau_b^M
\rangle\!\rangle},\qquad M=2^n,
\]
is infinite.
\end{lemma}

\begin{proof}
Choose lifts $u_1,\ldots,u_b$ of the $\tau_i$ in the free pro-$2$
group of the presentation. By \eqref{stabilitédegré}, the added
relations $u_i^M$ have depth at least $M$. The polynomial of the
resulting presentation therefore satisfies
\[
P_{G_M}(t)\leq P_G(t)+bt^M<0
\]
for $n$ large enough. Vinberg's criterion implies that $G_M$ is
infinite.
\end{proof}

\subsection{Pro-2 groups}
Let $\Pi$ be a pro-$2$ group generated by $2g$ elements. Define its Frattini subgroup $\Phi(\Pi)$ as the closure of the subgroup generated by the squares and the commutators
$$\Phi(\Pi):=\overline{\Pi^{2}[\Pi,\Pi]}.$$
We need the following lemma on pro-$2$ groups in the proof of Proposition \ref{tourgeometrique}. It provides a way to make a quotient map between pro-$2$ groups invariant under an action while controlling the number of relations and generators of the quotient in terms of the action.

\begin{lemma}\label{lemmaequivariantquotient}
    Let $\Pi$ be a pro-$2$ group generated by $2g$ elements. Let $H$ be a profinite group acting continuously on $\Pi$, with image $\Gamma\subset \mathrm{Aut}(\Pi)$ generated by at most $n$ elements. Assume that the action of $H$ is trivial on $\Pi/\Phi(\Pi)$. Then, for any surjective morphism
    $$q:\Pi\twoheadrightarrow F_g$$ into a free pro-$2$ group of rank $g$, there exists a quotient 
    $$F_g\twoheadrightarrow G$$ such that the induced surjective morphism
    $$p:\Pi\twoheadrightarrow G$$ is $H$-invariant and $G$ satisfies $r(G)\leq 2g\cdot n$ and $d(G)=g$.
\end{lemma}
\begin{proof}
    Let $x_1,...,x_{2g}$ be generators of $\Pi$, and let $\gamma_1,...,\gamma_n$ be generators of $\Gamma$. Define the pro-$2$ group $G$ as the quotient of $F_g$ by the relations 
    $$\{q(\gamma_j(x_i)x_i^{-1})\ ,\ 1\leq j\leq n,\ 1\leq i\leq 2g\}.$$
Denote by $$p:\Pi\rightarrow G$$ the induced surjective morphism. It is $H$-invariant by the relations imposed in $G$ and by the continuity of the action.

As the action of $\Gamma$ on $\Pi/\Phi(\Pi)$ is trivial, the elements $\gamma_j(x_i)x_i^{-1}$ are in the Frattini subgroup $\Phi(\Pi)$, hence their image $q(\gamma_j(x_i)x_i^{-1})$ is in the Frattini subgroup $\Phi(F_g)$. This implies that we have the inclusion
$$\mathrm{Ker}(F_g\rightarrow G)\subset \Phi(F_g).$$
Therefore, we have an isomorphism
$$F_g/\Phi(F_g)\simeq G/\Phi(G)$$
    and an equality $d(F_g)=d(G)=g$.
Finally, we imposed $2g\cdot n$ new relations in $G$; hence, the inequality $r(G)\leq 2g\cdot n$.
\end{proof}

\section{A geometric Golod--Shafarevich tower}
\subsection{Small Galois action and infinite pro-2 quotient}
Let $C$ be a smooth, projective, geometrically connected curve of genus $g\geq 2$ over $\Q$, endowed with a rational point $P$.

We want to provide a criterion to ensure that the étale fundamental group $\pi^{\text{ét}}_1(C,*)$ of $C$ has an infinite pro-$2$ quotient $$\pi^{\text{ét}}_1(C,*)\twoheadrightarrow G .$$
The idea is to use a specialization theorem 
$$\pi^{\text{ét}}_1(C_{\overline{\Q}},*)\twoheadrightarrow \pi^{\text{ét}}_1(\mathcal{C}_{\overline{\F}_2},*)$$
and to apply it to a curve admitting a proper semistable model $\mathcal{C}$ whose geometric special fiber at $2$ has only rational components. This would give a surjective map
$$\pi^{\text{ét}}_1(C_{\overline{\Q}},*)\twoheadrightarrow F_{g}$$ into a free pro-$2$ group of rank $g$. The rational point $P$ splits the exact sequence  \cite[Exp. IX, Th. 6.1 and Cor. 6.4]{SGA1-rev}
$$1\rightarrow \pi^{\text{ét}}_1(C_{\overline{\Q}},*)\rightarrow \pi^{\text{ét}}_1(C,*)\rightarrow G_{\Q}\rightarrow 1$$
where $G_{\Q}$ denotes the Galois group of $\Q$. Therefore, we have the decomposition 
$$\pi^{\text{ét}}_1(C,*)\simeq \pi^{\text{ét}}_1(C_{\overline{\Q}},*)\rtimes G_{\Q}$$ and we seek a quotient $F_g/H$ such that the composite surjective map $$\pi^{\text{ét}}_1(C_{\overline{\Q}},*)\twoheadrightarrow F_{g}\twoheadrightarrow F_g/H$$ extends to a surjective map
$$\pi^{\text{ét}}_1(C,*)\twoheadrightarrow F_{g}/H$$
by killing the action of the Galois group. To apply the Golod--Shafarevich theorem, we need to ensure that the following inequality holds 
$$r(F_{g}/H)<d(F_{g}/H)^{2}/4.$$

This is possible under the assumption that the action of the Galois group on 
$$H^{1}_{\text{ét}}(C_{\overline{\Q}}, \F_2)$$ is trivial and under some assumptions on the primes of bad reduction of $C$.
\subsection{The geometric construction}
\begin{proposition}\label{tourgeometrique}
    Let $C$ be a smooth, projective, geometrically connected curve of genus $g\geq 2$ over $\Q$, endowed with a rational point $P$. Assume the following conditions:
    \begin{enumerate}
        \item $C_{\Q_2}$ has a proper semistable model $\mathcal{C}_{\Z_2}$ over $\Z_2$ for which every geometric component of the special fiber is rational,
        \item the action of $G_\Q$ on $H^{1}_{\text{ét}}(C_{\overline{\Q}}, \F_2)$ is trivial,
        \item there exists a finite set of primes $S$ that contains $2$ and every prime of bad reduction of $C$, which satisfies $8(\text{Card} (S)+1)<g$.
    \end{enumerate}
    Then there exists a quotient $$\pi^{\text{ét}}_1(C,*)\twoheadrightarrow G$$ where $G$ is an infinite pro-$2$ group. Moreover, its restriction to the geometric fundamental group is surjective and factors through specialization at $2$, and the quotient map kills the image of the section associated to $P$.
\end{proposition}
\begin{proof}
    We denote by $$\Pi:=\pi^{\text{ét}}_1(C_{\overline{\Q}},*)^{(2)}$$ the pro-$2$ quotient of the geometric étale fundamental group. Grothendieck's specialization theorem gives a surjective map 
    $$\pi^{\text{ét}}_1(C_{\overline{\Q}},*)^{(2)}\twoheadrightarrow \pi^{\text{ét}}_1(\mathcal{C}_{\overline{\F}_2},*)^{(2)}$$
    between the pro-$2$ quotients.
    
    Let $\Delta$ be the dual graph of the geometric special fiber of $\mathcal{C}_{\Z_2}$, and let $\widetilde{C}_v$ denote the normalization of the component corresponding to a vertex $v$. The genus formula for a semistable curve gives
    $$g=\sum_{v}g(\widetilde{C}_v)+b_1(\Delta).$$
    By hypothesis, every $\widetilde{C}_v$ is rational, so that $b_1(\Delta)=g$. Since their normalizations have trivial geometric étale fundamental group, the pro-$2$ fundamental group of the special fiber is the free pro-$2$ group on the cycles of $\Delta$. Thus we have an isomorphism
    $$\pi^{\text{ét}}_1(\mathcal{C}_{\overline{\F}_2},*)^{(2)}\simeq F_g,$$
    where $F_g$ is a free pro-$2$ group of rank $g$, and hence a surjective morphism
    $$\Pi\twoheadrightarrow F_g.$$
Let $$\rho:G_{\Q}\rightarrow \mathrm{Aut}(\Pi)$$ be the pointed Galois action induced by the splitting associated with $P$. We want to study the image 
    $$\Gamma:=\rho(G_{\Q})$$ of the action and show that it is a pro-$2$ group with a controlled number of generators. In fact, we have a surjective morphism 
    \begin{equation}\label{actiongalois}
    G^{(2)}_{\Q,S}\twoheadrightarrow \Gamma
    \end{equation}
    where $G^{(2)}_{\Q,S}$ denotes the Galois group of the maximal pro-$2$ extension of $\Q$ unramified outside $S$. To see that $\Gamma$ is a pro-$2$ group, note that the $G_{\Q}$-equivariant isomorphism $$(\Pi/\Phi(\Pi))^\vee\simeq H^{1}_{\text{ét}}(C_{\overline{\Q}},\F_2)$$ and hypothesis (2) imply that the action of $G_{\Q}$ on $\Pi/\Phi(\Pi)$ is trivial. Hence the following inclusion holds
    $$\Gamma\subset{\text{Ker}}(\mathrm{Aut}(\Pi)\rightarrow\mathrm{Aut}(\Pi/\Phi(\Pi))).$$
    This implies that $\Gamma$ is a pro-$2$ group as a closed subgroup of a pro-$2$ group. At every prime outside $S$, the point $P$ extends, by properness, to a section of a smooth proper model of $C$. Smooth proper specialization for the pro-$2$ fundamental group therefore shows that this pointed action is unramified outside $S$. This gives the surjective map \ref{actiongalois}.
    
    The pro-$2$ group $$G^{(2)}_{\Q,S}$$ has $\text{Card}(S)+1$ generators. Indeed, since $2\in S$, Kummer theory identifies the vector space $H^1(G_{\mathbb Q,S}^{(2)},\mathbb F_2)$
with the subspace of $\Q^\times/(\Q^\times)^2$
having basis the classes of $-1$ and the primes in $S$.
Thus, by the Burnside basis theorem,
\[
d(G_{\mathbb Q,S}^{(2)})
=
\dim_{\mathbb F_2}H^1(G_{\mathbb Q,S}^{(2)},\mathbb F_2)
=
\operatorname{Card}(S)+1.
\]

    Hence, the pro-$2$ group $\Gamma$ with at most $\text{Card}(S)+1$ generators acts by automorphisms on the pro-$2$ group $\Pi$ which has $2g$ generators. By hypothesis, the action is trivial on $\Pi/\Phi(\Pi)$ and we may apply Lemma \ref{lemmaequivariantquotient} to the quotient $$\Pi\twoheadrightarrow F_{g}.$$

    Therefore, we have a $G_{\Q}$-invariant surjective map
    $$p:\Pi\twoheadrightarrow G$$ where $G$ is a pro-$2$ group with $g$ generators and at most $2g(\text{Card}(S)+1)$ relations. The assumption on $S$ gives
    $$r\leq 2g(\text{Card}(S)+1)<2g\cdot g/8=g^{2}/4=d(G)^{2}/4.$$ The Golod--Shafarevich criterion therefore implies that $G$ is infinite. The $G_{\Q}$-invariance of $p$ gives a surjective morphism
    $$\Pi\rtimes G_{\Q}\twoheadrightarrow G,\qquad (u,\sigma)\longmapsto p(u).$$
    Composing with the natural quotient from the semidirect product decomposition above gives a surjective morphism
    $$\pi^{\text{ét}}_1(C,*)\twoheadrightarrow G$$
    that kills the image of the section associated with $P$. Since $G$ is pro-$2$, this morphism factors through a surjective morphism
    $$\pi^{\text{ét}}_1(C,*)^{(2)}\twoheadrightarrow G.$$
\end{proof}
\section{An explicit curve}

Let $C$ be the smooth, projective, geometrically connected curve over $\Q$ defined by the equation 
\begin{equation}\label{courbeexplicite}
y^{2}=\underset{i=1}{\overset{m}{\prod}}(x-i)(x-i-8m)\end{equation}
where $m=2^{256}$. This is a hyperelliptic curve of genus $g=m-1$. Let $P=\infty^{+}$.

\subsection{Rationality of the 2-torsion of the Jacobian}
We have a canonical $G_\Q$-equivariant isomorphism
$$J_C[2](\overline{\Q})\simeq H^{1}_{\text{ét}}(C_{\overline{\Q}},\F_2)$$
where $J_C[2](\overline{\Q})$ denotes the geometric $2$-torsion points of the Jacobian $J_C$ of the curve $C$ defined by \ref{courbeexplicite}.

All the Weierstrass points of $C$ are defined over $\Q$, and their differences generate the $2$-torsion of the Jacobian. Hence, all the geometric points of $J_C[2]$ are defined over $\Q$ and the Galois action of $G_\Q$ on $J_C[2](\overline{\Q})$ is trivial so that the second assumption of Proposition \ref{tourgeometrique} is satisfied.

\subsection{Primes of good reduction}
If $i$ and $j$ denote two roots of the polynomial, then 
$$\lvert i-j\rvert\leq 9m$$ which implies that the primes of bad reduction of $C$ are all contained in the set 
$$S:=\{p\in \mathcal{P}, p\leq 9m\}.$$
In particular, using the estimates of Rosser and Schoenfeld \cite{RosserSchoenfeld1962}, $$8(\text{Card}(S)+1)\leq \frac{144m}{\log 9m}+8< \frac{144}{179}m+8<m-1=g$$ and the third assumption of Proposition \ref{tourgeometrique} is satisfied.

\subsection{A semistable model $\mathcal{C}_{\Z_2}$ of $C$ over $\Z_2$}

It remains to verify the first assumption of Proposition \ref{tourgeometrique}: the existence of a proper semistable model $\mathcal{C}_{\Z_2}$ of $C_{\Q_2}$ over $\Z_2$ with rational geometric components in its special fiber. We construct such a model whose dual graph consists of two copies of a tree $T$ joined by $m$ edges. 

\begin{proposition}\label{modelsemistable}
    The curve $C$ defined by \ref{courbeexplicite} admits a proper semistable model $\mathcal{C}_{\Z_2}$ over $\Z_2$. Every irreducible component of its special fiber is isomorphic to $\mathbb{P}^1_{\F_2}$, and the first Betti number of the dual graph of its special fiber is $g=m-1$.
\end{proposition}

\begin{proof}
Write $m=2^k$, with $k=256$.
Normalize the absolute value by $|2|_2=1/2$, and for $1\leq i\leq m$, let
\[
D_i:=\{|x-i|_2< m^{-1}\},
\qquad
\overline{D}_i:=\{|x-i|_2\le m^{-1}\},
\qquad
U:=\mathbb{P}_{\Q_2}^{1,\mathrm{an}}
  \setminus\bigcup_{i=1}^mD_i.
\]
The disks $\overline{D}_i$ are pairwise disjoint since
$|i-j|_2\ge 2/m$ for $i\ne j$.

We blow up $$\mathbb{P}^1_{\Z_2}$$ $k$
successive times as follows: at each time, blow up the distinct closed points
where the strict transforms of the sections $x=i$ meet the
special fiber. The resulting regular proper model $\mathcal B$
has special fiber a tree $T$ of projective lines over $\mathbb{F}_2$.
Its terminal components have coordinates $$t_i=(x-i)/m.$$
Removing the special points $t_i=0$ from its formal completion
gives a formal model $\mathfrak U$ of $U$.

Let $\mathcal{C}_{\Z_2}$ be the normalization of $\mathcal B$
in $\Q_2(C)$. It is a proper flat model
with generic fiber $C_{\Q_2}$. Since $\mathcal B$ is excellent, normalization
commutes with the formal completion considered here. We now describe
this completion.

On $U$, the inequalities $|8m/(x-i)|_2\le 1/8$ allow us
to define the convergent binomial series
\[
h_i(x)=\sqrt{1-\frac{8m}{x-i}}.
\]
Thus, the hyperelliptic cover splits over $U$ as
\[
y=\pm\prod_{i=1}^m(x-i)h_i(x).
\]
By the uniqueness of normalization, the model above $\mathfrak U$
is therefore the disjoint union of two copies of $\mathfrak U$.

On $\overline{D}_i$, the same square roots are defined for $j\ne i$.
Let
\[
w_i:=\frac{y}{m\prod_{j\ne i}(x-j)h_j(x)}.
\]
In the coordinate $w_i$, the equation is given by $w_i^2=t_i(t_i-8)$. The change of coordinates
\[
V_i=\frac{t_i-4+w_i}{2},
\qquad
W_i=\frac{t_i-4-w_i}{2}
\]
yields
\[
V_iW_i=4,\qquad t_i=V_i+W_i+4.
\]
Consequently, the normalization above the formal chart
$\operatorname{Spf}\Z_2\langle t_i\rangle$ is
\[
\operatorname{Spf}
\bigl(\Z_2\langle V_i,W_i\rangle/(V_iW_i-4)\bigr).
\]
Indeed, this ring is normal by Serre's criterion, is finite
over $\Z_2\langle t_i\rangle$ because
\[
V_i^2-(t_i-4)V_i+4=0,
\]
and has the required generic fiber.

Modulo $2$, this chart becomes $$V_iW_i=0,$$ with
$t_i=V_i+W_i$. On the boundary, where $t_i$ is invertible,
Hensel's lemma splits the cover into two branches.
They are exchanged by the hyperelliptic involution
$(V_i,W_i)\mapsto(W_i,V_i)$, so they attach to the two distinct
copies of $\mathfrak U$.
Each chart thus completes the two corresponding punctured
components and joins them at a split node.

Hence $\mathcal{C}_{\Z_2}$ is semistable, all components of its
special fiber are projective lines over $\mathbb{F}_2$, and
its dual graph consists of two copies of $T$ joined by $m$
edges. If $T$ has $v$ vertices, the dual graph of the special fiber has $2v$ vertices and $2(v-1)+m$ edges. Since it is connected, its first Betti number is \(m-1=g\).
\end{proof}

Therefore, we can apply Proposition \ref{tourgeometrique} to the curve $C$ defined by \ref{courbeexplicite}, which gives us an infinite pro-$2$ quotient
$$\pi^{\text{ét}}_1(C,*)\twoheadrightarrow G.$$

\section{Extending the covers over $\mathrm{Spec}\ \Z$}
We retain the notation of Proposition \ref{tourgeometrique}: let $G$ be the infinite pro-$2$ quotient of $\pi^{\text{ét}}_1(C,*)$. Each finite index subgroup of $G$ corresponds to a connected finite étale cover $C'\rightarrow C$ of $C$, and we want to extend each $C'\rightarrow C$ over $\Q$ to a connected finite étale cover $\mathcal{C}'\rightarrow \mathcal{C}^{\circ}$ over $\mathrm{Spec}\ \Z$. Here, $\mathcal{C}^{\circ}$ is defined as an open subscheme 
$$\mathcal{C}^{\circ}\hookrightarrow{\mathcal{C}}$$ of the proper, regular minimal model $\mathcal{C}$ of $C$ defined as 
$$\mathcal{C}^{\circ}:=\mathcal{C}\setminus \underset{p\in S\setminus\{2\}}{\bigcup}\underset{\substack{
E\in \operatorname{Irr}(\mathcal{C}_p)\\
s(p)\notin E
}}{\bigcup}E$$
which is flat over $\mathrm{Spec}\ \Z$. $s$ denotes the extension of the rational point $P\in C(\Q)$ to a section of the proper arithmetic surface $\mathcal{C}\rightarrow \mathrm{Spec}\ \Z$.

\begin{proposition}\label{étendrerevetements}
    Let $N$ be any finite index normal subgroup of $G$, and let $C'\rightarrow C$ be a connected finite étale cover of $C$ corresponding to $N$. Then this cover extends to a connected finite étale morphism $\mathcal{C}'\rightarrow \mathcal{C}^{\circ}$ where $\mathcal{C}'$ denotes the normalization of $\mathcal{C}^{\circ}$ in $k(C')$. Moreover, these extensions of covers over $\mathrm{Spec}\ \Z$ are compatible in the infinite pro-$2$ tower. 
\end{proposition}
Before proving the proposition, we deduce the following corollary.
\begin{corollary}
    The étale fundamental group of the regular, integral arithmetic surface $\mathcal{C}^{\circ}$ admits an infinite pro-$2$ quotient 
    $$\pi^{\text{ét}}_1(\mathcal{C}^{\circ},*)\twoheadrightarrow G.$$
\end{corollary}

We now prove Proposition \ref{étendrerevetements}. We proceed in three steps. First, we show that the étale covers of the generic fiber extend to étale covers of the model $\mathcal{C}_{\Z_2}$ constructed in Proposition \ref{modelsemistable}. Second, we show that these covers extend to étale covers of the base change to $\Z_2$ of the proper, regular minimal model of $C$. Finally, we extend these covers to étale covers of $\mathcal{C}^{\circ}$ using Abhyankar's lemma and the section $s$ extending $P$ at the primes of bad reduction.
\begin{proof}
Let $N\triangleleft G$ be a finite index normal subgroup.
The quotient $G/N$ determines a connected finite \'etale Galois
cover $C'\to C$. Let $$\mathcal{C}'\rightarrow \mathcal{C}^{\circ}$$ be the normalization of $\mathcal{C}^{\circ}$ in the function field $k(C')$ of $C'$. We have to show that the morphism $\mathcal{C}'\rightarrow \mathcal{C}^{\circ}$ is unramified in codimension~$1$.

Above $2$, consider the semistable model $\mathcal C_{\Z_2}$
of Proposition \ref{modelsemistable}, with special fiber $\mathcal C_{\mathbb{F}_2}$.
By construction, the geometric quotient
\[
\pi_1^{\text{ét}}(C_{\overline{\Q}_2},*)
\twoheadrightarrow G/N
\]
factors through
\[
\pi_1^{\text{ét}}(\mathcal C_{\overline{\mathbb{F}}_2},*)
\twoheadrightarrow G/N.
\]
Since specialization is surjective and compatible with the Galois
actions, this last quotient is $G_{\mathbb{F}_2}$-invariant.
Using the specialization of $P$, it therefore extends to a quotient
\[
\pi_1^{\text{ét}}(\mathcal C_{\mathbb{F}_2},*)
\twoheadrightarrow G/N
\]
which kills the specialized section.
The equivalence between finite étale covers of a proper scheme over
a henselian local ring and those of its special fiber yields a finite
étale cover of $\mathcal C_{\Z_2}$.
Its generic fiber is isomorphic to $C'_{\Q_2}$, since the
corresponding representations agree on the geometric fundamental
group and both kill the section associated with $P$.

To compare this model with the proper regular minimal model
$\mathcal C$, choose a regular proper model over $\Z_2$
dominating both $\mathcal C_{\Z_2}$ and
$\mathcal C\times_{\Z}\Z_2$.
The pullback of the preceding cover to this common model is finite
étale. The morphism to $\mathcal C\times_{\Z}\Z_2$
is an isomorphism over an open subset containing every
codimension-one point. Consequently, the normalization of
$\mathcal C\times_{\Z}\Z_2$ in the function field
of $C'_{\Q_2}$ is unramified in codimension one; hence, it is finite
étale by purity of the branch locus.
It follows that $\mathcal C'\to\mathcal C^{\circ}$ is unramified
above $2$.
\

If $p$ is an odd prime, let $E$ be the irreducible
component of $\mathcal C_p$ containing $s(p)$.
The morphism $\mathcal C\to\operatorname{Spec}\mathbf Z$ is smooth
at $s(p)$. Thus $E$ is the unique component through $s(p)$
and has multiplicity one. Let $A$ be the strict henselization of
$\mathcal O_{\mathcal C,s(p)}$.
The cover induced by $C'\to C$ over $\mathrm{Spec}\ A[1/p]$
is tamely ramified along the regular divisor defined by $p$,
since its Galois group $G/N$ is a $2$-group.
By Abhyankar's lemma, it is a disjoint union of covers of the form
\[
\mathrm{Spec}\ \bigl(A[1/p][z]/(z^e-p)\bigr),
\]
where $e$ is a power of $2$.
Pulling back along the section gives a trivial cover, because
$P$ splits completely in $C'\to C$.
However, $z^e-p$ is Eisenstein over the strictly henselian
discrete valuation ring of the section, so this is possible
only if $e=1$.
Consequently, $\mathcal C'\to\mathcal C^{\circ}$ is unramified
along $E$.
\

If $p\notin S$, good reduction implies that $E$ is the entire
fiber $\mathcal C_p$.
If $p\in S\setminus\{2\}$, it is the only component retained
in $\mathcal C^{\circ}_p$.
Together with the argument above $2$, this proves that the
normalization is unramified along every vertical divisor of
$\mathcal C^{\circ}$.
It is also unramified along every horizontal divisor, since
$C'\to C$ is étale.

Since $\mathcal C^{\circ}$ is regular, purity of
the branch locus now shows that this morphism is finite étale.
Moreover, $\mathcal C'$ is integral, hence connected.
Finally, the uniqueness of normalization makes these extensions
compatible with the transition morphisms between the finite
quotients of $G$.
\end{proof}

\section{A proper example}

The boundary $D^{\circ} := \mathcal{C}\setminus \mathcal{C}^{\circ}$ of the arithmetic surface $\mathcal{C}^{\circ}$ in the proper, regular
minimal model $\mathcal{C}$ of the curve $C$ defined by \ref{courbeexplicite} is supported above odd primes in
$S$. We can eliminate the inertia along this boundary by adding relations to the
infinite pro-$2$ group $G$: however, adding these relations could make the quotient
finite. Still, we can bound the inertia by taking powers of these relations: using
large enough powers, Vinberg's version of the Golod--Shafarevich theorem implies that
the quotient is infinite.

\begin{theorem}\label{properexample}
There exists a proper, regular, integral arithmetic surface
\[
f:X\longrightarrow\operatorname{Spec}\Z
\]
with a section, whose étale fundamental group admits an infinite
pro-$2$ quotient. Its generic fiber $X_\Q$ is a smooth, projective,
geometrically connected finite étale cover of the curve $C$ defined
by \eqref{courbeexplicite}.
\end{theorem}

\begin{proof}
By Section~5, we have a surjective morphism
\[
\pi_1^{\text{ét}}(\mathcal C^{\circ},*)\twoheadrightarrow G.
\]
The presentation constructed in Lemma~\ref{lemmaequivariantquotient}
has $g$ generators and at most $$2g\cdot(\operatorname{Card}(S)+1)$$
relations, all of depth at least $2$. Thus
\[
P_G(2/g)\leq -1+\frac{8(\operatorname{Card}(S)+1)}{g}<0.
\]

By resolving the boundary, choose a proper regular model
$\mathcal B\to\mathcal C$ which is an isomorphism over
$\mathcal C^{\circ}$ and for which the reduced boundary
\[
D:=\mathcal B\setminus\mathcal C^{\circ}=D_1\cup\cdots\cup D_b
\]
is a strict normal crossings divisor. It is still supported above
odd primes. The inertia images in $G$ along the $D_i$ are therefore
procyclic; choose generators $\tau_i$ of these images.
By Lemma~\ref{powersofinertia}, for $M=2^n$ large enough the quotient
\[
G_M=G/\overline{\langle\!\langle
\tau_1^M,\ldots,\tau_b^M
\rangle\!\rangle}
\]
is infinite. Its inertia groups along $D$ have orders dividing $M$.

Apply Lemma~\ref{finiteinertia} to this quotient. We obtain an open
normal subgroup $H$ of $G_M$ and a proper normal surface $Y$
with a surjection
\[
\pi_1^{\text{ét}}(Y,*)\twoheadrightarrow H.
\]
The group $H$ is infinite, since it has finite index in $G_M$.

Let $U\subset Y$ be the finite étale cover of $\mathcal C^{\circ}$
corresponding to $H$. Choose a resolution $X\to Y$ that is an
isomorphism over the regular open subscheme $U$. The pullbacks
of the covers $Y'\to Y$ are finite étale and remain connected,
since their generic fibers are connected. They therefore give
\[
\pi_1^{\text{ét}}(X,*)\twoheadrightarrow H.
\]
The surface $X$ is proper, regular, integral, and flat over $\mathrm{Spec}\ \Z$.

Finally, $P$ splits completely in every cover of the original
tower. The section $s$ therefore lifts to $U$, and hence to $X$.
The generic fiber $X_\Q$ is the finite étale cover of $C$
corresponding to $H$. It is geometrically connected because the
restriction of $\pi_1^{\text{ét}}(C,*)\twoheadrightarrow G_M$ to
the geometric fundamental group is surjective.
\end{proof}

Finally, we prove the lemma.

\begin{lemma}\label{finiteinertia}
Let $\mathcal B$ be a proper, regular, integral arithmetic surface,
and let $D$ be a reduced strict normal crossings divisor supported
above odd primes. Write $U:=\mathcal B\setminus D$, and let $Q$ be
a pro-$2$ quotient of $\pi_1^{\text{ét}}(U,*)$ whose inertia groups
along the components of $D$ are finite. For an open normal subgroup
$L$ of $Q$, denote by $Y_L$ the normalization of $\mathcal B$ in the
function field of the connected cover of $U$ corresponding to $L$.

There exists an open normal subgroup $H$ of $Q$ such that
$Y_K\to Y_H$ is finite étale for every open normal subgroup
$K$ of $Q$ contained in $H$. In particular,
\[
\pi_1^{\text{ét}}(Y_H,*)\twoheadrightarrow H.
\]
\end{lemma}

\begin{proof}
Choose a power $M$ of $2$ divisible by the orders of all the generic
inertia groups. For a geometric point $x$ of $D$, let
\[
U_x:=U\times_{\mathcal B}
\operatorname{Spec}\mathcal O_{\mathcal B,x}^{\mathrm{sh}}.
\]
If $r$ components of $D$ pass through $x$, Abhyankar's lemma
\cite[Exp.~XIII, App.~I, Prop.~5.2]{SGA1-rev} gives
\[
\pi_1^{\text{ét}}(U_x,*)^{(2)}\simeq\Z_2^{\,r},
\qquad r\leq 2,
\]
with generators corresponding to the inertia along these components.
Its image $J_x$ in $Q$ is thus a quotient of $(\Z/M\Z)^r$.
Up to conjugation, there are only finitely many such groups: the
generic inertia groups and the groups at the finitely many crossing
points. We can therefore choose an open normal subgroup $H$ of $Q$
such that $H\cap J_x=\{1\}$ for every $x$.

For $K\subset H$ as in the statement, the maps from
$\pi_1^{\text{ét}}(U_x,*)$ to $Q/K$ and $Q/H$ have the same kernel.
Consequently, over $U_x$ the corresponding covers have isomorphic
connected components. Normalization commutes with strict
henselization, so $Y_K\to Y_H$ becomes a disjoint union of
isomorphisms locally above $x$. It is therefore finite étale,
including above the crossing points. These connected covers form
a compatible tower and give the asserted quotient.
\end{proof}

\bibliographystyle{alpha}
\bibliography{biblio_corrigee}
\end{document}